\documentclass[10pt]{amsart}

\usepackage[utf8]{inputenc}
\usepackage[english]{babel}
\usepackage{amssymb, amsmath, amsfonts}
\usepackage{xcolor}
\usepackage{tikz}
\usetikzlibrary{arrows,calc}

\newtheorem{theo}{Theorem}[section]
\newtheorem{lem}[theo]{Lemma}
\newtheorem{pro}[theo]{Proposition}
\newtheorem{cor}[theo]{Corollary}
\newtheorem{prob}[theo]{Problem}
\newtheorem{fact}[theo]{Fact}

\theoremstyle{definition}
\newtheorem{defi}[theo]{Definition}

\newtheorem{rem}[theo]{Remark}

\numberwithin{equation}{section}

\newcommand{\retY}{\ensuremath{\mathcal{K}(Y)}}
\newcommand{\retX}{\ensuremath{\mathcal{K}(X^*)}}
\newcommand{\completionXY}{\ensuremath{\widehat{X}_Y}}

\newcommand{\pty}[1]{\ensuremath{(\mathcal{#1})}}
\newcommand{\ptyE}{\pty{E}}
\newcommand{\ptyEE}{\pty{E'}}
\newcommand{\ptyC}{\pty{C}}

\newcommand{\ptyLK}{\pty{LK}}

\newcommand{\w}{\omega}

\newcommand{\FF}{\mathcal{F}}
\newcommand{\N}{\mathbb{N}}

\def\({\left(}
\def\){\right)}

\title{Completeness in the Mackey topology by norming subspaces}

\author{A.J. Guirao}
\address{Instituto Universitario de Matem\'{a}tica Pura y Aplicada\\ Universitat Polit\`{e}cnica de Val\`{e}ncia\\ Camino de Vera s/n\\ 46022 Valencia\\ Spain}
\email{anguisa2@mat.upv.es}

\author{G. Mart\'{\i}nez-Cervantes}
\address{Dpto. de Matem\'{a}ticas, Facultad de Matem\'{a}ticas, Universidad de Murcia, 30100 Espinardo (Murcia), Spain}
\email{gonzalomartinezcervantes@gmail.com}

\author{J. Rodr\'{i}guez}
\address{Dpto. de Ingenier\'{i}a y Tecnolog\'{i}a de Computadores,
Facultad de Inform\'{a}tica, Universidad de Murcia, 30100 Espinardo (Murcia), Spain}
\email{joserr@um.es}



\thanks{A.J. Guirao was supported by projects MTM2017-83262-C2-1-P (AEI/FEDER, UE) 
and 19368/PI/14 (Fundaci\'on S\'eneca). G. Mart\'{\i}nez-Cervantes and J. Rodr\'{i}guez
were supported by projects MTM2014-54182-P and
MTM2017-86182-P (AEI/FEDER, UE) and 19275/PI/14 (Fundaci\'on S\'eneca).}

\dedicatory{Dedicated to the memory of Bernardo Cascales}

\begin{document}
\begin{abstract}
We study the class of Banach spaces $X$ such that the locally convex space $(X,\mu(X,Y))$ is complete for every
norming and norm-closed subspace $Y \subset X^*$, where $\mu(X,Y)$ denotes the Mackey topology on $X$ associated to the dual pair $\langle X,Y\rangle$.
Such Banach spaces are called fully Mackey complete. We show that fully Mackey completeness is implied by Efremov's property~($\mathcal{E}$) 
and, on the other hand, it prevents the existence of subspaces isomorphic to~$\ell_1(\omega_1)$. 
This extends previous results by Guirao, Montesinos and Zizler [J. Math. Anal. Appl. 445 (2017), 944--952] and Bonet and Cascales [Bull. Aust. Math. Soc. 81 (2010), 409-413]. 
Further examples of Banach spaces which are not fully Mackey complete are exhibited, like $C[0,\omega_1]$ and the long James space $J(\omega_1)$. 
Finally, by assuming the Continuum Hypothesis, we construct a Banach space with $w^*$-sequential dual unit ball which is not fully Mackey complete. 
A key role in our discussion is played by the (at least formally) smaller class of Banach spaces $X$ such that $(Y,w^*)$ has the Mazur property
for every norming and norm-closed subspace $Y \subset X^*$.  
\end{abstract}

\maketitle

\section{Introduction}

Let $X$ be a Banach space and $Y\subset X^*$ a $w^*$-dense subspace (not necessarily norm-closed). 
The Mackey topology $\mu(X,Y)$ on~$X$ associated to the dual pair $\langle X, Y\rangle$ is the locally convex topology
of uniform convergence on elements of the family
\[
	\retY:=\{K\subset Y\colon \, K \text{ is absolutely convex and } w^*\text{-compact}\}.
\]
Several authors have recently discussed the {\em completeness} of $(X,\mu(X,Y))$,
see \cite{bon-cas}, \cite{gui-mon} and~\cite{gui-mon-ziz}. This research line was motivated initially by Kunze's paper~\cite{kun} on vector integration (cf. \cite{and-zie,rod18}). 
Bonet and Cascales~\cite{bon-cas} exploited some results of~\cite{sua-san1} to prove that if $X$ contains a subspace isomorphic to~$\ell_1(\mathfrak{c})$, then there is a norming and norm-closed subspace $Y \subset X^*$ for which $(X,\mu(X,Y))$ is {\em not} complete. At this point we stress
that, in general, the completeness of $(X,\mu(X,Y))$ implies that $Y$ is norming, see \cite[Proposition~3]{gui-mon-ziz}. Guirao, Montesinos and Zizler~\cite{gui-mon-ziz}
exhibited a connection between the completeness of $(X,\mu(X,Y))$ and the {\em Mazur property} of $(Y,w^*)$
(i.e. the property that every $w^*$-sequentially continuous linear functional $f: Y \to \mathbb{R}$ is $w^*$-continuous). More precisely:
	\begin{itemize}
		\item[(a)] If $(Y,w^*)$ has the Mazur property and $Y$ is norm-closed, then $(X,\mu(X,Y))$ is complete, see \cite[Proposition~10]{gui-mon-ziz}.
		\item[(b)] If $(X,\mu(X,Y))$ is complete and every $K\in\retY$ is Fr\'{e}chet-Urysohn, then $(Y,w^*)$ has the Mazur property, see \cite[Proposition~1]{gui-mon-ziz}.
	\end{itemize}
	
Let us introduce a couple of definitions:	
\begin{defi}
	A Banach space $X$ is said to be
	\emph{fully Mackey complete} (resp. \emph{fully Mazur}) if $(X,\mu(X,Y))$ is complete
	(resp. $(Y,w^*)$ has the Mazur property) for every norming and norm-closed subspace $Y\subset X^*$.
\end{defi}
Thus, statement~(a) above implies that every fully Mazur space is fully Mackey complete. A sufficient condition on a Banach space~$X$ to be fully Mazur
is that $(B_{X^*},w^*)$ is Fr\'{e}chet-Urysohn (see \cite[Theorem~5]{gui-mon-ziz}), which includes the case of weakly compactly generated
spaces and, more generally, weakly Lindelof determined ones. On the other hand, the aforementioned result of~\cite{bon-cas}
says that a fully Mackey complete Banach space cannot contain subspaces isomorphic to~$\ell_1(\mathfrak{c})$.

In this paper we go a bit further in studying fully Mazur and fully Mackey complete Banach spaces. The paper is organized as follows.

Section~\ref{section:preliminaries} introduces the basic terminology and contains some preliminary known results on the completeness of Mackey topologies. 
In addition, we prove that completeness and quasi-completeness are equivalent for $(X,\mu(X,Y))$ whenever $Y$ is 
norm-closed (Proposition~\ref{prop:q-comp_equals_comp}). 

We begin Section~\ref{section:2} by showing that in statement~(b) above it is enough to assume (besides the completeness of~$(X,\mu(X,Y))$) that 
every convex $w^*$-sequentially closed subset of any~$K\in \retY$ is $w^*$-closed
(Proposition~\ref{prop-Sc-CompleteImpliesMazur}). This is a localization of the Banach space {\em property~$\ptyEE$} studied in~\cite{avi-mar-rod,gon3} (which means that 
every convex $w^*$-sequentially closed bounded subset of the dual is $w^*$-closed). In particular, fully Mackey completeness is equivalent to being fully Mazur
for Banach spaces with property $\ptyEE$. We stress that property~$\ptyEE$ is strictly weaker than having Fr\'{e}chet-Urysohn dual ball, as
witnessed by the so-called Johnson-Lindentrauss spaces (see~\cite[Theorem~3.1]{gon3}).

Our main results in Section~\ref{section:2} characterize fully Mazur and fully Mackey complete Banach spaces.
Write $S_1(A) \subset X^*$ to denote the set of all limits of $w^*$-convergent 
sequences contained in the set~$A\subset X^*$. In Theorem~\ref{profullyMazurnorminglysequential} we prove that a Banach space~$X$ is fully Mazur if and only if $S_1(Y)=X^*$ 
for every norming and norm-closed subspace~$Y \subset X^*$. As a consequence, every Banach space having Efremov's property~$\ptyE$
is fully Mazur (Corollary~\ref{cor:Efremov}). Recall that~$X$ is said to have {\em Efremov's property~$\ptyE$} 
if $S_1(C)=\overline{C}^{w^*}$ for every convex bounded set $C\subset X^*$ (see~\cite{pli3}). The following implications hold in general:
\begin{center}
$(B_{X^*},w^*)$ is Fr\'{e}chet-Urysohn $\Longrightarrow$
$X$ has property~$\ptyE$ $\Longrightarrow$ $X$ has property~$\ptyEE$.
\end{center}
Under the Continuum Hypothesis there exist Banach spaces separating the three conditions above (see \cite{avi-mar-rod}), 
while it is unknown whether such examples exist in ZFC. On the other hand, in Theorem~\ref{theofullyMackeyLK}
we characterize fully Mackey completeness in a similar spirit, namely, a Banach space~$X$
is shown to be fully Mackey complete if and only if for every norming and norm-closed subspace $Y \subset X^\ast$ and every $x^\ast \in X^\ast\setminus Y$ there is
$K \in \retX$ such that $K \subset Y \oplus [x^\ast ]$ and $x^\ast \in \overline{K \cap Y}^{w^\ast}$.

Section~\ref{section:3} is mostly devoted to showing further examples of Banach spaces which are {\em not} fully Mackey complete.
Theorem~\ref{TheoGeneralizationNonfullyMackey} provides a technical tool which applies to prove that spaces like $\ell_1(\omega_1)$ and $C[0,\omega_1]$ 
fail to be fully Mackey complete. In particular, since this property is inherited by closed subspaces (Corollary~\ref{profullyMackeyhereditary}), it follows
that a fully Mackey complete Banach space cannot contain subspaces isomorphic to~$\ell_1(\omega_1)$, thus improving
the result of~\cite{bon-cas} which was mentioned above. The absence of subspaces isomorphic to~$\ell_1$ is not sufficient 
for fully Mackey completeness, as the example of $C[0,\omega_1]$ makes clear. On the other hand, we also investigate fully Mackey completeness
within the setting of {\em dual} Banach spaces. It is shown that if $X^*$ is fully Mackey complete, then $X$ is $w^*$-sequentially dense in~$X^{**}$
(Theorem~\ref{theofullyCorsondual}). As a consequence, we include a sharp characterization of the fully Mackey 
completeness of~$X^*$, in some particular cases, in terms of the compact topological space $(B_{X^{**}},w^*)$ 
(Corollaries~\ref{cor:DualOfSeparable} and~\ref{cor:DualOfAsplund}). 

One may wonder whether property $\ptyEE$ implies fully Mackey completeness. 
We will show that this is not the case. By modifying a construction of~\cite{avi-mar-rod}, under the Continuum Hypothesis,
we provide an example of a maximal almost disjoint family $\mathcal{F}$ of infinite subsets of~$\mathbb{N}$
for which the Banach space $C(K_\mathcal{F})$ is not fully Mackey complete, where
$K_{\mathcal{F}}$ is the Stone space of the Boolean algebra generated by $\mathcal{F}$ and the finite subsets of~$\N$ (Theorem~\ref{theo:EprimaNoFullyCompleto}). Note that, without any extra set-theoretic assumption, all Banach spaces of the form $C(K_\mathcal{F})$ have property~$\ptyEE$  (see~\cite{gon3}).
We finish the paper by collecting several open problems. For instance, we do not know whether fully Mazur and fully
Mackey completeness are equivalent properties.

\section{Terminology and preliminaries}\label{section:preliminaries}

All our topological spaces are Hausdorff and all our linear spaces are real. 
Given a linear space $E$, we denote by $E^{\#}$ the linear space
consisting of all linear functionals from~$E$ to~$\mathbb{R}$. For any set $S \subset E$, the symbol $[S]$ stands for the 
subspace of~$E$ generated by~$S$. Given a dual pair $\langle E,F \rangle$, we denote
by $w(E,F)$ and $w(F,E)$ the induced weak topologies on~$E$ and~$F$. When $E=X$ is a Banach
space and $F=X^*$ (its topological dual), we simply write $w=w(X,X^*)$ and $w^*=w(X^*,X)$. A 
locally convex space~$E$ is said to have the {\em Mazur property} if every
sequentially continuous element of~$E^{\#}$ is continuous. A topological space $T$
is said to be {\em Fr\'{e}chet-Urysohn} if, for each $B \subset T$, any element of $\overline{B}$ 
is the limit of a sequence contained in~$B$. A subset $C$ of a topological space~$T$ is said to be {\em sequentially closed}
if no sequence in~$C$ converges to a point in~$T\setminus C$. 
Given a Banach space~$X$, we write
$B_X=\{x\in X:\|x\|\leq 1\}$ (the closed unit ball of~$X$). A subspace $Y \subset X^*$ 
is said to be {\em norming} if the formula 
$$
	|||x|||=\sup\{x^*(x):\, x^*\in Y \cap B_{X^*}\}, \quad x\in X,
$$
defines an equivalent norm on~$X$. Given a compact topological space~$K$, we denote by $C(K)$
the Banach space of all real-valued continuous functions on~$K$, equipped with the supremum norm. For each $t\in K$,
we write $\delta_t\in C(K)^*$ to denote the evaluation functional at~$t$, i.e. $\delta_t(h):=h(t)$
for all $h\in C(K)$.
 
Throughout this paper $X$ is a Banach space. Given a $w^*$-dense subspace $Y\subset X^*$, we consider
the subspace of~$Y^{\#}$ defined by
\[
	\completionXY:=\{f\in Y^{\#}\colon \, f|_K\,\text{ is } w^*\text{-continuous for every } K\in\retY\}.
\]
Note that $X$ can be identified with the subspace of~$\completionXY$ consisting of all $w^*$-continuous elements of~$Y^{\#}$, that is,
for any $w^*$-continuous $f\in Y^{\#}$ there is a unique $x\in X$ such that $\langle x, y^* \rangle=f(y^*)$ for all $y^*\in Y$.
Observe that $\langle \completionXY,Y \rangle$ is a dual pair 	and that an absolutely convex set $K\subset Y$ is $w^*$-compact if and only if it is 
$w(Y,\completionXY)$-compact. 	In particular, the restriction of $\mu(\completionXY,Y)$ to~$X$ coincides with~$\mu(X,Y)$.
Grothendieck's characterization of the completion of a locally convex space (see e.g. \cite[\S 21.9]{kot}), when applied to our setting, yields the following:

\begin{fact}\label{fact:G}
Let $Y \subset X^*$ be a $w^*$-dense subspace.
\begin{enumerate}
\item[(i)] $(\completionXY,\mu(\completionXY,Y))$ is the completion of $(X,\mu(X,Y))$.
\item[(ii)] $(X,\mu(X,Y))$ is complete if and only if every element of~$\completionXY$ is \\ $w^*$-continuous.
\end{enumerate}
\end{fact}

The following result is extracted from the proof of \cite[Proposition~10]{gui-mon-ziz}:

\begin{fact}\label{lem_CompImplyCont}
Let $Y \subset X^*$ be a $w^*$-dense and norm-closed subspace. Then 
$$
	\completionXY \subset \{f\in Y^{\#}: \, f \mbox{ is }w^*\mbox{-sequentially continuous}\} \subset Y^*.
$$
\end{fact}

Thus, under the additional assumption that $Y$ is norm-closed, the Hahn-Banach theorem
guarantees that for every $f\in \completionXY$ there is some $x^{**}\in X^{**}$ such that $x^{**}|_Y=f$.
Since $\{x^{**}\in X^{**}: \, x^{**}|_Y \mbox{ is }w^*\mbox{-continuous}\}=X\oplus Y^\perp$, we get:

\begin{fact}\label{fact:XplusYperp}
	Let $Y \subset X^*$ be a $w^*$-dense and norm-closed subspace. Then
	$(X,\mu(X,Y))$ is complete if and only if 
	$$
		\{x^{**}\in X^{**}: \, x^{**}|_K \mbox{ is }w^*\mbox{-continuous for every }K\in \retY\}=X\oplus Y^\perp.
	$$
\end{fact}

The following useful fact (see \cite[Lemma~11]{gui-mon-ziz}) will be needed several times.

\begin{fact}\label{fact:11}
Let $Y \subset X^*$ be a norming subspace. If $x^{**}\in X^{**}\setminus (X\oplus Y^\perp)$, then
$Y \cap \ker(x^{**})$ is norming as well.
\end{fact}

A general locally convex space very often lacks completeness but sometimes it satisfies a weaker property, called quasi-completeness, that is enough for major applications of completeness (Krein-Smulyan theorem, for instance). Recall that a locally convex space~$E$ is said to be \emph{quasi-complete} if every bounded and closed subset of~$E$ is complete. 
We next show that in our setting quasi-completeness and completeness coincide.

\begin{pro} \label{prop:q-comp_equals_comp}
	Let $Y\subset X^*$ be a $w^*$-dense and norm-closed subspace. Then $(X,\mu(X,Y))$ is quasi-complete if and only if it is complete.
\end{pro}

\begin{proof}
	Take $A:=B_X$ and apply the bipolar theorem (see e.g. \cite[Theorem~3.38]{fab-ultimo}) 
	in the dual pair $\langle \completionXY, Y\rangle$ to obtain
	\[
		\overline{A}^{w(\completionXY,Y)}=A^{\circ\circ}=\{f\in\completionXY: \, f(y^*)\leq 1 \mbox{ for all } y^*\in B_Y\}.
	\]
	Bearing in mind Mazur's theorem (see e.g. \cite[Theorem~3.45]{fab-ultimo}), we deduce that
	\begin{equation}\label{eqn:Mazur}
		\overline{A}^{\mu(\completionXY,Y)}=\overline{A}^{w(\completionXY,Y)}=\{f\in\completionXY: \, f(y^*)\leq 1 \mbox{ for all } y^*\in B_Y\}.
	\end{equation}
	
	Suppose $(X,\mu(X,Y))$ is quasi-complete. We will show that $(X,\mu(X,Y))$ is complete by applying Fact~\ref{fact:G}. 
	Take any $f \in \completionXY$. Since $f \in Y^*$ (by Fact~\ref{lem_CompImplyCont}), 
	we can assume that $f\in B_{Y^*}$ (normalize!). Then $f\in \overline{A}^{\mu(\completionXY,Y)}$ (by~\eqref{eqn:Mazur})
	and so there is a net $\left(x_\alpha\right)_{\alpha\in \Lambda}$ in~$A$ which $\mu(\completionXY,Y)$-converges to~$f$. In particular, $\left(x_\alpha\right)_{\alpha\in \Lambda}$ 
	is a bounded Cauchy net in the quasi-complete locally convex space $(X,\mu(X,Y))$. Then 
	$\left(x_\alpha\right)_{\alpha\in \Lambda}$ is $\mu(X,Y)$-convergent to some $x\in X$  and so $f=x\in X$.
\end{proof}

\section{Mazur property and Mackey completeness}\label{section:2}

The following proposition improves statement~(b) in the introduction:

\begin{pro}\label{prop-Sc-CompleteImpliesMazur}
	Let $Y\subset X^*$ be a $w^*$-dense subspace such that: 
	\begin{enumerate}
	\item[(i)] $(X,\mu(X,Y))$ is complete, 
	\item[(ii)] every $K\in\retY$ has the following property: every convex $w^*$-sequentially closed subset of~$K$ is $w^*$-closed. 
	\end{enumerate}
	Then $(Y,w^*)$ has the Mazur property.
\end{pro}
\begin{proof}
	Let $f: Y \to \mathbb{R}$ be linear and $w^*$-sequentially continuous. Since 
	$(X,\mu(X,Y))$ is complete, in order to prove that $f$ is $w^*$-continuous it suffices to check that $f|_K$ is $w^*$-continuous
	for any $K\in \retY$ (Fact~\ref{fact:G}). Clearly, for a given $K\in \retY$, the $w^*$-continuity of~$f|_K$ is equivalent to 
	\begin{enumerate}
	\item[($\star$)] $f^{-1}(C)\cap K$ is $w^*$-closed for every {\em convex} closed set $C\subset \mathbb{R}$.
	\end{enumerate}   
	Since $f$ is linear, for any convex $C \subset \mathbb{R}$ the set $f^{-1}(C)\cap K$ is convex and so 
	it is $w^*$-closed if and only if it is $w^*$-sequentially closed (by~(ii)). Therefore, the $w^*$-sequential continuity of~$f$
	ensures that ($\star$) holds and the proof is finished.
\end{proof}

\begin{cor}\label{cor:EE}
Suppose $X$ has property~$\ptyEE$. 
\begin{enumerate}
\item[(i)] Let $Y \subset X^*$ be a $w^*$-dense and norm-closed subspace. Then $(X,\mu(X,Y))$ is complete if and only if $(Y,w^*)$ has the Mazur property.
\item[(ii)] $X$ is fully Mackey complete if and only if it is fully Mazur.
\end{enumerate}
\end{cor}

The characterizations of fully Mazur and fully Mackey complete Banach spaces given in the next two theorems are
our main results in this section.

\begin{theo}
	\label{profullyMazurnorminglysequential}
	The following statements are equivalent:
	\begin{enumerate}
	\item[(i)] $X$ is fully Mazur.
	\item[(ii)] $S_1(Y)=X^*$ for every norming and norm-closed subspace~$Y \subset X^*$. 
	\end{enumerate}
\end{theo}

\begin{proof} (i)$\Rightarrow$(ii) 
	Suppose condition (ii) fails and fix a norming and norm-closed subspace $Z\subset X^*$ such that $X^*\setminus S_1(Z)\neq \emptyset$. 
	Fix $x^*\in X^*\setminus S_1(Z)$, take the subspace $Y:=Z\oplus [x^*]\subset X^*$ (which is norming and norm-closed) and the functional $f \in Y^*$ defined by 
	$$
		f(z^*+\lambda x^*):=\lambda
		\quad\mbox{ for all }z^*\in Z \mbox{ and }\lambda\in\mathbb{R}.
	$$ 
	Observe that $f$ is not $w^*$-continuous because $x^*\in X^*=\overline{Z}^{w^*}$, $f(x^*)=1$ and $\ker(f)=Z$.
	
	Let us show that $f$ is $w^*$-sequentially continuous. Let $\(y_n^*\)_{n\in \mathbb{N}}$ be a sequence in~$Y$ 
	which $w^*$-converges to some $y^*\in Y$. Write 
	$$
		y^*=z^*+\lambda x^*
		\quad\mbox{and}\quad
		y_n^*=z_n^*+\lambda_n x^*
	$$ 
	for some $z^*,z_n^*\in Z$ and $\lambda,\lambda_n\in \mathbb{R}$. Then
	$$
		\((z_n^*-z^*)+(\lambda_n-\lambda)x^*\)_{n\in \mathbb{N}} 
	$$ 
	is $w^*$-null. Since $z_n^*-z^* \in Z$ for all $n\in \mathbb{N}$ and $x^*\not \in S_1(Z)$, we conclude that $f(y_n^*)=\lambda_n \to f(y^*)=\lambda$
	as $n\to \infty$. This proves that $f$ is $w^*$-sequentially continuous. We have shown that $(Y,w^*)$ fails the Mazur property 
	and therefore that $X$ is not fully Mazur.
	
	(ii)$\Rightarrow$(i) Let $Y\subset X^*$ be a norming and norm-closed subspace. To prove that $(Y,w^*)$ has the Mazur property,
	take a $w^*$-sequentially continuous $f\in Y^{\#}$. Since $f\in Y^{*}$ (Fact~\ref{lem_CompImplyCont}), there is $x^{**}\in X^{**}$ such that $x^{**}|_Y=f$. By contradiction, 
	suppose $f$ is not $w^*$-continuous. Then $x^{**} \in X^{**}\setminus (X\oplus Y^\perp)$ and we can consider the norming 
	and norm-closed subspace $Z:=Y\cap {\rm ker}(x^{**})\subset X^*$ (Fact~\ref{fact:11}). 
	Condition~(ii) applied to~$Z$ ensures that $S_1(Z)=X^*$ and so the $w^*$-sequential continuity of~$f$ implies 
	that $Y=Z$, a contradiction which finishes the proof. 
\end{proof}

As an application, we generalize the result that Banach spaces having Fr\'{e}chet-Urysohn dual ball are fully Mazur (see \cite[Theorem~5]{gui-mon-ziz}):

\begin{cor}\label{cor:Efremov}
If $X$ has Efremov's property~$\ptyE$, then $S_1(Y)=X^*$ for every norming subspace $Y \subset X^*$. Consequently, $X$ is fully Mazur.
\end{cor}
\begin{proof}
Let $Y\subset X^*$ be any norming subspace. By the Hahn-Banach
separation theorem, we have $\overline{Y\cap B_{X^*}}^{w^*} \supset \delta B_{X^*}$ for some $\delta>0$.
On the other hand, $\overline{Y\cap B_{X^*}}^{w^*}=S_1(Y\cap B_{X^*})$ (since $X$ has property~$\ptyE$)
and therefore $S_1(Y)=X^*$. Theorem~\ref{profullyMazurnorminglysequential} now applies to deduce that $X$ is fully Mazur.
\end{proof}

\begin{theo}
	\label{theofullyMackeyLK}
	$X$ is fully Mackey complete if and only if the following condition holds:
	\begin{enumerate}
	\item[$\ptyLK$] For every norming and norm-closed subspace $Y \subset X^\ast$ and every $x^\ast \in X^\ast\setminus Y$ there is
	$K \in \retX$ such that $K \subset Y \oplus [x^\ast ]$ and $x^\ast \in \overline{K \cap Y}^{w^\ast}$.
	\end{enumerate} 
\end{theo}

\begin{proof}
	Suppose $X$ fails condition $\ptyLK$. Take a norming and norm-closed subspace $Y_0\subset X^\ast$ and $x^\ast \in X^\ast \setminus Y_0$ such that for every 
	$K \in \retX$ we have
	\begin{equation}\label{eqn:laimpli}
		x^\ast \in \overline{K \cap Y_0}^{w^\ast}
		\quad \Longrightarrow \quad 
		K \nsubseteq Y_0 \oplus [x^\ast ].
	\end{equation}
	Set $Y:=Y_0 \oplus  [x^\ast]$ and define $f\in Y^*$ by declaring $f(y_0^*+\lambda x^\ast):=\lambda $ for every $y_0^* \in Y_0$ 
	and $\lambda \in \mathbb{R}$. 	Note that $f$ is not $w^*$-continuous (because $x^*\in X^*=\overline{Y_0}^{w^\ast}$, $f(x^*)=1$ and $\ker(f)=Y_0$).
	Thus, in order to prove that $(X,(\mu(X,Y))$ is not complete it is enough to show that $f|_{K'}$ is $w^*$-continuous for every $K' \in \retY$
	(Fact~\ref{fact:G}).
	
	By contradiction, suppose $f|_{K'}$ is not $w^*$-continuous. Then there is a net $\(x_\alpha^*= y_\alpha^\ast+ \lambda_\alpha x^\ast\)_{\alpha\in\Lambda}$
	in~$K'$ (where $y_\alpha^\ast\in Y_0$ and $\lambda_\alpha\in \mathbb{R}$) which $w^*$-converges to some $y^\ast+\lambda x^\ast \in K'$ 
	(where $y^\ast\in Y_0$ and $\lambda\in \mathbb{R}$) and such that $\(\lambda_\alpha\)_{\alpha\in\Lambda}$ does not converge to~$\lambda$. 
	Since $K'$ is bounded, so is $\(\lambda_\alpha\)_{\alpha\in\Lambda}$. Fix 
	$M>0$ such that $|\lambda_\alpha| \leq M$ for all $\alpha\in \Lambda$. By passing to a subnet if necessary we can assume that 
	$\(\lambda_\alpha\)_{\alpha\in\Lambda}$ converges to some $\lambda' \neq \lambda$. Set	
	$$ 
		K := \frac{K'+M\|x^\ast\|B_{[x^\ast]}+\|y^\ast\| B_{[y^\ast]}}{\lambda-\lambda'} \subset Y=Y_0\oplus[x^*]
	$$
	and notice that $K \in \retX$ since it is a sum of absolutely convex and $w^*$-compact sets.
	Moreover, the net
	$$ 
		\(\frac{y^*_\alpha-y^*}{\lambda-\lambda'}\)_{\alpha\in\Lambda}=\(\frac{x^*_\alpha-\lambda_\alpha x^\ast-y^*}{\lambda-\lambda'}\)_{\alpha\in\Lambda}
	$$
	is contained in~$K\cap Y_0$ and $w^*$-converges to~$x^*$, which 
	contradicts~\eqref{eqn:laimpli}. This shows that $(X,(\mu(X,Y))$ is not complete.	
	
	Conversely, we now prove that condition~$\ptyLK$ implies that $X$ is fully Mackey complete. The argument
	is similar to the proof of (ii)$\Rightarrow$(i) in Theorem~\ref{profullyMazurnorminglysequential}.
	Let $Y \subset X^*$ be a norming and norm-closed subspace and take any $f\in \completionXY$.
	Then $f\in Y^*$ (Fact~\ref{lem_CompImplyCont}) and so $f=x^{**}|_Y$ for some $x^{**}\in X^{**}$. If $f$ is not $w^*$-continuous, then $x^{**}\not\in X\oplus Y^\perp$
	and therefore $Y_0:=Y \cap \ker(x^{**})=\ker(f)$ is norming (Fact~\ref{fact:11}). Pick $x^*\in Y\setminus Y_0$. 
	Condition~$\ptyLK$ applied to $Y_0$ and $x^*$ ensures the existence of $K \in \retX$ such that $K \subset Y_0 \oplus [x^\ast] \subset Y$ and 
	$x^\ast \in \overline{K\cap Y_0}^{w^\ast}$. This contradicts the $w^*$-continuity of~$f|_K$, because $f$ vanishes on $K \cap Y_0$ and $f(x^\ast) \neq 0$.
	It follows that $f$ is $w^*$-continuous. This shows that $(X,\mu(X,Y))$ is complete (by Fact~\ref{fact:G}).
	\end{proof}	

\begin{cor}
\label{profullyMackeyhereditary}
	If $X$ is fully Mazur (resp. fully Mackey complete), then any closed subspace of~$X$
	is fully Mazur (resp. fully Mackey complete).
\end{cor}
\begin{proof} Let $X_0 \subset X$ be a a closed subspace and denote by $r:X^*\to X_0^*$ the bounded linear operator defined by 
$r(x^*):=x^*|_{X_0}$ for every $x^*\in X^*$. Given any norm-closed subspace $Y_0\subset X_0^*$,
the norm-closed subspace $Y:=r^{-1}(Y_0)\subset X^*$ is norming (for~$X$) whenever $Y_0$ is norming (for~$X_0$), see e.g. \cite[p.~269, Exercise 5.6]{fab-ultimo}. 
The conclusion now follows at once from Theorems~\ref{profullyMazurnorminglysequential}
and~\ref{theofullyMackeyLK}, bearing in mind the $w^*$-$w^*$-continuity of~$r$.
\end{proof}

\section{Banach spaces which are not fully Mackey complete}\label{section:3}

The following technical result provides a sufficient condition on a Banach space to fail fully Mackey completeness.
Recall that a topological space is said to be {\em countably compact} if every sequence in it has a cluster point.

\begin{theo}
	\label{TheoGeneralizationNonfullyMackey}	
	Let $T$ be a countably compact topological space with a distinguished point $\infty \in T$. Suppose there is a function $f\colon T \to X$ satisfying: 
	\begin{enumerate}
	\item[(i)] $f(\infty)=0$; 
	\item[(ii)] $Y:=\lbrace x^* \in X^*: \, x^* \circ f \mbox{ is continuous}\rbrace$ is norming and norm-closed; 
	\item[(iii)] there exist $\varepsilon>0$ and $x_\infty^* \in X^*$ such that $D:=\lbrace t \in T: x_\infty^*(f(t))> \varepsilon \rbrace$ 
	intersects every $\mathcal{G}_\delta$-set containing~$\infty$.
	\end{enumerate}
	Then $X$ is not fully Mackey complete.
\end{theo}

\begin{proof}
	By Theorem \ref{theofullyMackeyLK}, it is enough to check that $X$ does not have property~$\ptyLK$. 
	Note that $x_\infty^*\circ f$ is not continuous at~$\infty$ and so $x_\infty^*\not\in Y$.
	Let $K$ be a bounded subset of~$Y \oplus [x^*_\infty]$ with $x^*_\infty \in  \overline{K \cap Y}^{w^\ast}$. 
	We will prove that $\overline{K \cap Y}^{w^\ast}$ is not contained in $Y \oplus [x^*_\infty]$ and, therefore,
	$K$ is not $w^*$-compact.
	
	To this end, we first construct by induction a sequence $\(t_n\)_{n\in \mathbb{N}}$ in~$D$, a sequence $\(x_n^*\)_{n\in\mathbb{N}}$ in~$K\cap Y$ and a decreasing
	sequence $\(U_n\)_{n\in \mathbb{N}}$ of open neighborhoods of~$\infty$ such that, for each $n\in \mathbb{N}$, we have:
	\begin{enumerate}
		\item[($a_n$)] $x_n^*(f(t_j))>\varepsilon$ for every $j\leq n$;
		\item[($b_n$)] $|x_n^*(f(t))| \leq \frac{\varepsilon}{n}$ for every $t\in \overline{U_n}$;
		\item[($c_n$)] $t_n\in U_{n-1}$ (with the convention $U_0:=T$).
	\end{enumerate} 
	
	For the first step, take any $t_1 \in D$. Since $x^*_\infty \in  \overline{K \cap Y}^{w^\ast}$, we can pick $x_1^* \in K \cap Y$ 
	such that $x_1^*(f(t_1))>\varepsilon$. 
	By the continuity of $x_1^* \circ f$, there is an open neighborhood $U_1$ of~$\infty$ such that $|x_1^*(f(t))| \leq \varepsilon$ for every $t\in \overline{U_1}$. 
	Suppose now that, for some $n\in \mathbb{N}$, we have already chosen $t_1, t_2, \dots , t_n \in D$, $x_1^*, x_2^*, \dots, x_n^* \in K \cap Y$ and
	$U_1 \supset U_2 \supset \dots \supset U_n$ open neighborhoods of $\infty$ such that ($a_i$), ($b_i$) and~($c_i$) hold for every $i \leq n$.
	Pick an arbitrary $t_{n+1} \in D\cap U_n$ and choose $x_{n+1}^* \in K\cap Y$ with $x_{n+1}^*(f(t_j))>\varepsilon$ 
	for every $j \leq n+1$ (bear in mind that the $t_i$'s belong to~$D$ and $x^*_\infty \in  \overline{K \cap Y}^{w^\ast}$). 
	Now, the continuity of $x_{n+1}^* \circ f$ ensures the existence of an open neighborhood $U_{n+1}$ of $\infty$ 
	contained in~$U_n$ such that $|x_{n+1}^*(f(t))| \leq \frac{\varepsilon}{n+1}$ for every $t\in \overline{U_{n+1}}$. This finishes the inductive construction.
	
	Let $x^*\in \overline{K\cap Y}^{w^*}$ be any $w^*$-cluster point of the sequence $\(x_n^*\)_{n\in \mathbb{N}}$. Then
	\begin{enumerate}
		\item[(a)] $x^*(f(t_n))\geq \varepsilon$ for every $n\in \mathbb{N}$;
		\item[(b)] $x^*(f(t))=0$ for every $t \in \bigcap_{n\in \mathbb{N}} \overline{U_n}$.
	\end{enumerate}
	We claim that $x^* \notin Y \oplus [x_\infty^*]$. 
	Our proof is by contradiction. Suppose $x^*+\lambda x_\infty^* \in Y$ for some $\lambda \in \mathbb{R}$. 
	By (a) and~(b), $x^*\circ f$ is not continuous at any cluster 
	point of the sequence $\(t_n\)_{n\in\mathbb{N}}$ (such cluster points exist since $T$ is countably compact and, by construction, they are contained in $\bigcap_{n\in \mathbb{N}} \overline{U_n}$), hence $\lambda \neq 0$.
	Observe that $|(x^*+\lambda x_\infty^*)(f(t))|> |\lambda| \varepsilon$ for every $t \in D\cap \bigcap_{n\in \mathbb{N}} U_n$. 
	But $\infty \in \overline{D\cap \bigcap_{n\in \mathbb{N}} U_n}$ (because $D$ intersects every $\mathcal{G}_\delta$-set containing $\infty$) and therefore  
	$(x^*+\lambda x_\infty^*)\circ f$ cannot be continuous at~$\infty$, a contradiction. 	
\end{proof}

The following corollary was already shown in \cite[Corollary~5(ii)]{gui-mon}. Here we prove it via the unifying approach of Theorem~\ref{TheoGeneralizationNonfullyMackey}.

\begin{cor}
	\label{corl1uncountable}
	$\ell_1(\Gamma)$ is not fully Mackey complete whenever $\Gamma$ is uncountable.
\end{cor}

\begin{proof}
	Let $T:=\Gamma \cup \{\infty\}$ be the one-point compactification of the set $\Gamma$ equipped with the discrete topology.
	Define $f:T \to \ell_1(\Gamma)$ by declaring $f(\gamma):=e_\gamma$ for all $\gamma\in \Gamma$ and $f(\infty):=0$.
	Then $\{x^*\in \ell_\infty(\Gamma): x^*\circ f \mbox{ is continuous}\}=c_0(\Gamma)$ is norming and norm-closed.
	Take $x_\infty^* := \chi_\Gamma \in \ell_\infty (\Gamma)$ and fix any $0<\varepsilon<1$. 
	Since $\Gamma$ is uncountable, $\{\infty\}$ is not a $\mathcal{G}_\delta$-set 
	and so $\lbrace t\in T: x_\infty^*(f(t))> \varepsilon \rbrace=\Gamma$ intersects every $\mathcal{G}_\delta$-set
	containing~$\infty$. The result now follows from Theorem~\ref{TheoGeneralizationNonfullyMackey}.	
\end{proof}

By putting together Corollaries~\ref{profullyMackeyhereditary} and~\ref{corl1uncountable}, we get:

\begin{cor}\label{corNol1}
If $X$ is fully Mackey complete, then it contains no subspace isomorphic to~$\ell_1(\omega_1)$.
\end{cor}

\begin{cor}
	$C([0,\w_1])$ is not fully Mackey complete.
\end{cor}
\begin{proof}
	Let $T:=[0,\w_1]$ and $\infty:=\w_1$. Define $f:T \to C([0,\w_1])$ by $f(\alpha):=\chi_{(\alpha, \w_1]}$ for all $\alpha<\omega_1$
	and $f(\infty):=0$. The subspace 
	$$
		Y:=\{x^*\in C([0,\w_1])^*: \, x^*\circ f \mbox{ is continuous}\}
	$$
	is norm-closed (bear in mind that $f$ is bounded) and norming, because it contains the set $\{\delta_{\beta+1}:\beta<\omega_1\}$. 
	Take $x_\infty^* := \delta_{\w_1}$ and fix any $0<\varepsilon<1$. 
	Then the set $\{t\in T: x_\infty^*(f(t))>\varepsilon\}=[0,\omega_1)$ intersects every $\mathcal{G}_\delta$-set containing~$\infty$ 
	(since $\{\infty\}$ is not a $\mathcal{G}_\delta$-set).  The result now follows from Theorem \ref{TheoGeneralizationNonfullyMackey}.
\end{proof}

We now focus on dual Banach spaces. Since $X$ is a norm-closed subspace of~$X^{**}$
which is norming for~$X^*$, Theorem~\ref{profullyMazurnorminglysequential} implies that $X$
is $w^*$-sequentially dense in~$X^{**}$ whenever $X^*$ is fully Mazur. In fact, we have the following:

\begin{theo}
\label{theofullyCorsondual}
If $X^*$ is fully Mackey complete, then $X$ is $w^*$-sequentially dense in~$X^{**}$.
\end{theo}

\begin{proof}
By Corollary~\ref{corNol1}, $X^*$ contains no subspace isomorphic to~$\ell_1(\omega_1)$, which 
implies that $X$ contains no subspace isomorphic to~$\ell_1$ (see e.g. \cite[Proposition~4.2]{van}).

Fix $x^{**} \in X^{**}\setminus X$. Since $X$ is a norming (for~$X^*$) and norm-closed subspace of~$X^{**}$, 
Theorem \ref{theofullyMackeyLK} ensures the existence of $K \in  \mathcal{K}(X^{**})$ such that 
$K \subset X \oplus [x^{**}]$ and $x^{**} \in \overline{K \cap X}^{w^\ast}$.
This implies that $K \cap X$ is not weakly compact and therefore it is not weakly sequentially compact (due to the Eberlein-\v{S}mulian theorem). 
Take any sequence $\left( x_n\right)_{n\in \N}$ in $K \cap X$ without weakly convergent subsequences. 
Since $X$ does not contain subspaces isomorphic to~$\ell_1$, we can suppose without loss of generality that 
$\left( x_n\right)_{n\in \N}$ is weakly Cauchy, thanks to Rosenthal's $\ell_1$-theorem (see e.g. \cite[Proposition~4.2]{van}). 
Therefore, $\left( x_n\right)_{n\in \N}$ is $w^*$-convergent to an element of~$X^{**}$
of the form $x+\lambda x^{**}$ with $x\in X$ and $\lambda \neq 0$ (since $K \subset X \oplus [x^{**}]$). 
Thus, $\left( \frac{x_n-x}{\lambda}\right)_{n\in \N}$ is a sequence in $X$ which $w^*$-converges to~$x^{**}$. This shows
that $X$ is $w^*$-sequentially dense in~$X^{**}$.
\end{proof}

Banach spaces which are $w^*$-sequentially dense in their bidual have been widely studied in the literature. 
We next include some related remarks on fully Mackey complete dual spaces which follow from Theorem~\ref{theofullyCorsondual}.

\begin{rem}\label{rem:longJames}
Every $w^*$-sequentially continuous linear functional $f\colon X^{**} \rightarrow \mathbb{R}$ is norm-continuous
when restricted to~$X$, i.e. $f|_X\in X^*$. Therefore, the equality $f(x^{**})=\langle x^{**},f|_X \rangle$ holds
for every $x^{**}\in X^{**}$ which is the $w^*$-limit of a sequence contained in~$X$. 
It follows that if $X$ is $w^*$-sequentially dense in~$X^{**}$, then $(X^{**},w^*)$ has the Mazur property.
This provides new non-trivial examples of Banach spaces which are not fully Mackey complete, 
such as the {\em long James space} $J(\omega_1)$ (see \cite{edg3}). Indeed, $J(\omega_1)$ is the dual of a Banach space~$X$
which is not $w^*$-sequentially dense in~$X^{**}$, since $(J(\omega_1)^*,w^*)$ fails the Mazur property.
\end{rem}

\begin{rem}\label{rem:dual}
If $X$ is $w^*$-sequentially dense in~$X^{**}$, then $X$ contains no subspace isomorphic to~$\ell_1$
(see e.g. \cite[Proposition~3.9]{rod-ver}) and, moreover, in each of the following particular cases $(B_{X^{**}},w^*)$ is Fr\'{e}chet-Urysohn:
\begin{itemize}
\item $X$ is {\em separable}, by the Odell-Rosenthal and Bourgain-Fremlin-Talagrand theorems (see e.g. \cite[Theorem~4.1]{van}). 
\item $X$ is {\em Asplund}, in fact, in this case $X^*$ is weakly Lindelof determined, i.e. $(B_{X^{**}},w^*)$ is Corson
(see \cite[Theorem~III-4]{dev-god} and \cite[Corollary~8]{ori1}).
\end{itemize}
\end{rem}

As a consequence:

\begin{cor}\label{cor:DualOfSeparable}
Suppose $X$ is separable. Then $X^*$ is fully Mackey complete if and only if $(B_{X^{**}},w^*)$ is Fr\'{e}chet-Urysohn
if and only if $X$ contains no subspace isomorphic to~$\ell_1$.
\end{cor}

\begin{cor}\label{cor:DualOfAsplund}
Suppose $X$ is Asplund. Then $X^*$ is fully Mackey complete if and only if $X^*$ is weakly Lindelof determined.
\end{cor}

By an {\em almost disjoint family} we mean an infinite family of pairwise almost disjoint infinite subsets of~$\N$, where 
two sets are said to be {\em almost disjoint} if they have finite intersection. 
For any almost disjoint family~$\FF$, we denote by $K_\FF$ the Stone (compact topological) space associated to the Boolean algebra generated by~$\FF$
and the finite subsets of~$\mathbb{N}$. Notice that there is a natural decomposition 
$$
	K_\FF=\mathbb{N}\cup \{ u_N: N \in \FF \}\cup \{\infty\},
$$ 
where each point of~$\mathbb{N}$ is isolated, the basic open neighborhoods of each $u_N$ are of the form $\{u_N\}\cup (N\setminus F)$ 
where $F \subset \mathbb{N}$ is finite, and the basic open neighborhoods of~$\infty$ are of the form $K_\FF \setminus \bigcup_{N\in \FF_0} (\{u_{N}\}\cup N)$
where $\FF_0 \subset \FF$ is finite. Then $K_\FF$ is scattered (of height~$3$) and so
$$
	C(K_\FF)^*= \ell_1(K_\FF)=
	\ell_1(\mathbb{N})\oplus\ell_1(\FF) \oplus [\delta_{\infty}].
$$ 
Observe that $\ell_1(\mathbb{N})$ is a norming (since $\mathbb{N}$ is dense in~$K_\FF$) and norm-closed subspace of~$C(K_\FF)^*$.
Under the Continuum Hypothesis, the construction in \cite[Section~4]{avi-mar-rod} provides a maximal (with respect to inclusion) almost disjoint family~$\FF$ 
for which no sequence in the convex hull of~$\{\delta_{n}: n\in \N\}$ is $w^*$-convergent to~$\delta_\infty$. We will improve such construction as follows:

\begin{theo}\label{theo:EprimaNoFullyCompleto}
Under the Continuum Hypothesis, there exists a maximal almost disjoint family $\FF$ such that: 
\begin{enumerate}
\item[(i)] No sequence in $\ell_1(\N)$ is $w^*$-convergent to~$\delta_\infty$.
\item[(ii)] $C(K_{\FF})$ is not fully Mackey complete.
\end{enumerate}
\end{theo}

Part~(i) will be proved with the help of Lemmas~\ref{lem:Gonzalo} and~\ref{CoroMatrices} below. 
The first one is a refinement of \cite[Lemma~4.2]{avi-mar-rod}:

\begin{lem}\label{lem:Gonzalo}
	Let $\FF=\{N_r: r\in \N\}$ be a countable almost disjoint family and let $\(\lambda_{i,j}\)_{i,j \in \N}\in \mathbb{R}^{\mathbb{N}\times \mathbb{N}}$ 
	be a matrix satisfying the following properties:
	\begin{enumerate}
	\item[(i)] $\lim_{i\to \infty} \lambda_{i,j} =0$ for every $j\in \N$;
	\item[(ii)] $\sum_{j\in \N}|\lambda_{i,j}|<\infty$ for every $i\in \N$; 
	\item[(iii)] $\lim_{i\to \infty} \sum_{j \in \N} \lambda_{i,j}=1$. 
	\end{enumerate} 
	If 
	\begin{equation}\label{eqn:sums}
		\lim_{i\to \infty} \sum_{j \in N_1 \cup \ldots \cup N_r} \lambda_{i,j}=0
		\quad
		\mbox{for every }
		r \in \N,
	\end{equation}
	then there exists an infinite set $N' \subset \N$ such that $\FF \cup \{ N' \}$ is almost disjoint and 
	$$
		\limsup_{i\to\infty} \sum_{j \in N'} \lambda_{i,j} \geq \frac{1}{2}.
	$$
\end{lem}
\begin{proof}
	Write $\tilde{N}_r:=N_1\cup\ldots\cup N_r$ for all $r\in \N$. 
	We will construct by induction a sequence $\(F_r\)_{r\in \N}$ of finite subsets of~$\N$ and a strictly increasing sequence $\(n_r\)_{r\in \N}$ of natural numbers
	as follows. Take any finite set $F_1 \subset \N$ and any $n_1\in \N$. Given $r\in \N$, $r\geq 2$, 
	suppose the finite sets $F_1,\dots,F_{r-1}\subset \N$  and $n_1<\dots<n_{r-1}$ in~$\N$ 
	have already been chosen. By~\eqref{eqn:sums} and (iii), we can find $n_r \in \N$ with $n_r>n_{r-1}$ in such a way that
	\begin{enumerate}
		\item[(a)] $\sum_{j \in \tilde{N}_r} \lambda_{n_r,j} \leq \frac{1}{16}$;
		\item[(b)] $\sum_{j \in\N} \lambda_{n_r,j} \geq \frac{3}{4}$.
	\end{enumerate} 
	Since $F_1\cup\ldots\cup F_{r-1}$ is finite, (i) allows us to assume further that
	\begin{enumerate}
		\item[(c)] $\sum_{j\in F_1\cup\ldots \cup F_{r-1}}|\lambda_{n_r,j}| \leq \frac{1}{16}$. 
	\end{enumerate}
	By~(ii), there is a finite set $F_r \subset \N \setminus \tilde{N}_r$ satisfying 
	\begin{equation}\label{eqn:absoluto}
		\sum_{j \in \N \setminus (\tilde{N}_r \cup F_r)} |\lambda_{n_r,j}|\leq \frac{1}{16}. 
	\end{equation}
	Notice that (a), (b) and \eqref{eqn:absoluto} yield
	\begin{equation}\label{eqn:sum}
		\sum_{j \in F_r} \lambda_{n_r,j} = 
		\sum_{j \in\N} \lambda_{n_r,j} - \sum_{j \in \tilde{N}_r} \lambda_{n_r,j} - 
		\sum_{j \in \N \setminus (\tilde{N}_r \cup F_r)} \lambda_{n_r,j}  \geq  \frac{3}{4}-\frac{1}{16}-\frac{1}{16}=\frac{5}{8}.
	\end{equation}
	This finishes the inductive construction.

	Let us check that $N' := \bigcup_{r \in \N} F_r$ satisfies the required properties. On one hand, for each $r\in \N$ with $r\geq 2$ we have
	$$
		N'\setminus F_r \subset \big(F_1\cup\ldots\cup F_{r-1}\big) \cup \big(\N\setminus (\tilde{N}_r\cup F_r)\big), 
	$$
	hence 
	$$
		\sum_{j\in N'\setminus F_r}\lambda_{n_r,j} +
		\sum_{j\in F_1\cup\ldots\cup F_{r-1}}|\lambda_{n_r,j}|+ \sum_{j\in \N\setminus (\tilde{N}_r\cup F_r)}|\lambda_{n_r,j}| \geq  0.
	$$
	This inequality, (c), \eqref{eqn:absoluto} and~\eqref{eqn:sum} yield
	\begin{multline*}
		\sum_{j\in N'} \lambda_{n_r,j}= \sum_{j\in F_r}\lambda_{n_r,j}+\sum_{j\in N'\setminus F_r} \lambda_{n_r,j} \\ \geq
		\sum_{j\in F_r}\lambda_{n_r,j} -\sum_{j\in F_1\cup\ldots\cup F_{r-1}}|\lambda_{n_r,j}|- \sum_{j\in \N\setminus (\tilde{N}_r\cup F_r)}|\lambda_{n_r,j}|
		\geq \frac{5}{8}-\frac{1}{16}-\frac{1}{16}=\frac{1}{2}.
	\end{multline*}
	As $r\in \mathbb{N}$ is arbitrary, it follows that
	$$
		\limsup_{i\to \infty} \sum_{j \in N'} \lambda_{i,j} \geq \frac{1}{2}.
	$$ 
	On the other hand, (i) ensures that $\lim_{i\to\infty} \sum_{j \in F} \lambda_{i,j} =0$ whenever $F \subset \N$ is finite, therefore $N'$ is infinite. By
	construction, $N'\cap N_1 \subseteq F_1$ and for each $r\in \N$ with $r\geq 2$ the intersection 
	$N' \cap N_r$ is contained in the finite set $F_1 \cup F_2 \cup \ldots \cup F_{r-1}$. 
	This shows that $\FF \cup \{N'\}$ is an almost disjoint family.
\end{proof}

\begin{lem}
	\label{CoroMatrices}
	Let $\{\(\lambda_{i,j}^\alpha\)_{i,j\in \N}: \alpha < \omega_1\}$ be a family of matrices of~$\mathbb{R}^{\mathbb{N}\times \mathbb{N}}$ 
	satisfying properties (i), (ii) and (iii) of Lemma~\ref{lem:Gonzalo}. 	Then there exists an almost disjoint family $\FF$ such that for every $\alpha < \omega_1$ 
	there is $N'_\alpha \in \FF$ for which the sequence $\(\sum_{j \in N'_\alpha} \lambda_{i,j}^\alpha\)_{i\in \N}$ does not converge to~$0$.
\end{lem}
\begin{proof}
Let $\mathcal{G}=\{N_r:r\in \N\}$ be any countable almost disjoint family. If there is $r\in \N$
for which $\(\sum_{j \in N_r} \lambda_{i,j}^0\)_{i\in \N}$ does not converge to~$0$, then we set $\FF_0:=\mathcal{G}$
and $N_0':=N_r$.
Otherwise, $\lim_{i\to \infty} \sum_{j \in N_r} \lambda_{i,j}^0 =0$ for all $r\in \N$. Observe that for each $r\in \N$ with $r\geq 2$ we have
$$
	\sum_{j \in N_1 \cup \ldots \cup N_r} \lambda^0_{i,j} =
	\sum_{j \in N_1 \cup \ldots \cup N_{r-1}} \lambda^0_{i,j}+\sum_{j \in N_r} \lambda^0_{i,j}
	-\sum_{j \in N_r \cap  (N_1 \cup \ldots \cup N_{r-1})} \lambda^0_{i,j}
$$
for all $i\in \N$ and
$$
	\lim_{i\to \infty} \sum_{j \in N_r\cap (N_1 \cup \ldots \cup N_{r-1})} \lambda^0_{i,j}=0,
$$
by property~(i) and the finiteness of $N_r\cap (N_1 \cup \ldots \cup N_{r-1})$. This clearly implies (by induction on~$r$) that
$$		
		\lim_{i\to \infty} \sum_{j \in N_1 \cup \ldots \cup N_r} \lambda^0_{i,j}=0
		\quad
		\mbox{for every }
		r \in \N,
$$
so Lemma~\ref{lem:Gonzalo} can be applied to find an infinite set $N'_0 \subset \N$ for which $\mathcal{F}_0:=\mathcal{G}\cup\{N'_0\}$
is an almost disjoint family and $\(\sum_{j \in N'_0} \lambda_{i,j}^0\)_{i\in\N}$ does not converge to~$0$. 

We now construct, by transfinite induction on~$\alpha<\omega_1$, an increasing chain $\(\FF_\alpha\)_{\alpha<\omega_1}$ of 
countable almost disjoint families and sets $N'_\alpha\in \FF_\alpha$ for which 
the sequence $\(\sum_{j \in N'_\alpha} \lambda_{i,j}^\alpha\)_{i\in \N}$ does not converge to~$0$. 
Suppose that $0<\alpha < \omega_1$ and that $\FF_\beta$ and $N'_\beta$ are already constructed for every $\beta<\alpha$. 
If 
\begin{equation}\label{eqn:4J}
	\lim_{i\to \infty} \sum_{j \in N} \lambda_{i,j}^\alpha=0
	\quad\mbox{for every }N \in \bigcup_{\beta<\alpha}\FF_\beta,
\end{equation}
then we can apply the argument above to $\bigcup_{\beta<\alpha}\FF_\beta$ (which is a countable almost disjoint family) and the 
matrix~$\(\lambda_{i,j}^\alpha\)_{i,j\in \N}$
in order to get an infinite set $N'_\alpha \subset \N$ such that $\FF_{\alpha}:=( \bigcup_{\beta<\alpha}\FF_\beta ) \cup \{N'_\alpha\}$ is
almost disjoint and $\(\sum_{j \in N'_\alpha} \lambda^\alpha_{i,j}\)_{i\in\N}$ does not converge to~$0$.  
Otherwise, if \eqref{eqn:4J} fails, then we take $\FF_{\alpha}:=\bigcup_{\beta<\alpha}\FF_\beta$
and any $N'_\alpha\in\FF_\alpha$ witnessing the failure of~\eqref{eqn:4J}.

Clearly, $\FF := \bigcup_{\beta<\w_1}\FF_\beta$ is an almost disjoint family satisfying the required property.
\end{proof}

\begin{proof}[Proof of Theorem~\ref{theo:EprimaNoFullyCompleto}]
	(i) The set of {\em all} matrices of~$\mathbb{R}^{\mathbb{N}\times \mathbb{N}}$ 
	satisfying properties (i), (ii) and (iii) of Lemma~\ref{lem:Gonzalo} has cardinality~$\mathfrak{c}$ and so, under the Continuum Hypothesis,
	it can be enumerated as $\{\(\lambda_{i,j}^\alpha\)_{i,j\in \N}: \alpha < \w_1\}$.
	Let $\FF$ be the almost disjoint family given by Lemma~\ref{CoroMatrices}. To check that $\FF$ is maximal, take any 
	infinite set $N=\{n_k:k\in \N\} \subset \N$ and define a matrix $\(\lambda_{i,j}\)_{i,j \in \N}\in \mathbb{R}^{\N\times \N}$ 
	by declaring $\lambda_{i,j}:=1$ if $n_i=j$ and $\lambda_{i,j}:=0$ otherwise. Obviously, it satisfies properties
	(i), (ii) and (iii) of Lemma~\ref{lem:Gonzalo}, hence there is $N' \in \FF$ such that $\(\sum_{j \in N'} \lambda_{i,j}\)_{i\in \N}$ does not converge to~$0$, 
	which clearly implies that $N \cap N' $ is infinite. This shows that $\FF$ is maximal.
	
	Suppose $\(x_i^*\)_{i\in\N}$ is a sequence in~$\ell_1(\N)$ which $w^*$-converges to~$\delta_\infty$ in~$C(K_{\FF})^*$. 
	For each $i\in\N$ we write $x_i^*=\sum_{j\in \N} \lambda_{i,j} \delta_j$, where $\lambda_{i,j}\in \mathbb{R}$ and
	$\sum_{j\in \N}|\lambda_{i,j}|<\infty$. Then 
	$$
		\lim_{i\to\infty} \lambda_{i,j} = \lim_{i\to\infty}x_i^*(\chi_{\{j\}})=\delta_\infty(\chi_{\{j\}})=0 \quad
		\mbox{for every }j\in \N
	$$
	and 
	$$
		\lim_{i\to\infty} \sum_{j \in \N } \lambda_{i,j} = 
		\lim_{i\to\infty} x_i^*(\chi_{K_\FF})=\delta_\infty(\chi_{K_\FF})=1. 
	$$
	Therefore, $\(\lambda_{i,j}\)_{i,j\in \N}=\(\lambda^\alpha_{i,j}\)_{i,j\in \N}$ for some $\alpha< \omega_1$. 
	But then there exists $N'_\alpha \in \FF$ such that $\(\sum_{j \in N'_\alpha} \lambda_{i,j}\)_{i\in \N}$ does not converge to~$0$,
	which is a contradiction since 
	$$ 
		\sum_{j \in N'_\alpha} \lambda_{i,j} =x_i^*(\chi_{\{u_{N'_\alpha}\}\cup N'_\alpha}) \quad\mbox{for all }i\in \N
	$$
	and $\delta_\infty(\chi_{\{u_{N'_\alpha}\}\cup N'_\alpha})=0$.

(ii) The space $C(K_\FF)$ is not fully Mazur by~(i) and Theorem~\ref{profullyMazurnorminglysequential}. On the other hand, 
since $K_\FF$ is scattered of countable height, $(B_{C(K_\FF)^*},w^*)$ is sequential, meaning that every $w^*$-sequentially closed
subset of $B_{C(K_\FF)^*}$ is $w^*$-closed (see \cite[Theorem~3.2]{gon3}). Hence
$C(K_\FF)$ has property~$\ptyEE$ and so $C(K_\FF)$ is not fully Mackey complete 
(apply Corollary~\ref{cor:EE}). 
The proof is finished.
\end{proof}

We finish the paper with some open questions:

\begin{prob}\label{prob:equivalence}
		Are fully Mazur and fully Mackey completeness equivalent? 
\end{prob}

As we pointed out in Corollary~\ref{cor:EE}, Problem~\ref{prob:equivalence} has an affirmative answer for Banach spaces with property~$\ptyEE$.

\begin{prob}
		Does fully Mackey completeness imply property~$\ptyEE$ or the weaker Corson's property~$\ptyC$?
\end{prob}

\begin{prob}\label{prob:sequentiallydense}
		Is $X^*$ fully Mackey complete whenever $X$ is $w^*$-sequentially dense in~$X^{**}$?
\end{prob}

Remark~\ref{rem:dual} makes clear that a negative answer to Problem~\ref{prob:sequentiallydense}
would be based on a non-separable and non-Asplund space~$X$ without subspaces isomorphic to~$\ell_1$.

\subsection*{Acknowledgements}
The authors wish to thank A. Avil\'{e}s for valuable discussions on the topic of this paper.
A.J. Guirao was supported by projects MTM2017-83262-C2-1-P (AEI/FEDER, UE) 
and 19368/PI/14 (Fundaci\'on S\'eneca). G. Mart\'{\i}nez-Cervantes and J. Rodr\'{i}guez
were supported by projects MTM2014-54182-P and
MTM2017-86182-P (AEI/FEDER, UE) and 19275/PI/14 (Fundaci\'on S\'eneca).

\bibliographystyle{amsplain}

\end{document}